\documentclass{amsart}
\usepackage{amssymb,xspace}
\usepackage[small]{diagrams}
\usepackage{hyperref}

\def\makedef#1#2{\expandafter\gdef\csname#1\endcsname{#2}}
\def\makelet#1#2{\expandafter\let\csname#1\expandafter\endcsname\csname#2\endcsname}
\def\mat#1{\ensuremath{#1}\xspace}
\def\defmath#1#2{\makedef{#1}{\mat{#2}}}
\def\redefmath#1{\makelet{temp@#1}{#1}\defmath{#1}{\csname temp@#1\endcsname}}

\def\defbb#1{\defmath{c#1}{\mathbb{#1}}}		
\def\defcal#1{\defmath{l#1}{\mathcal{#1}}}	
\def\deffrak#1{\defmath{g#1}{\mathfrak{#1}}} 
\def\bbcal#1{\defbb{#1}\defcal{#1}\deffrak{#1}}

\bbcal A
\bbcal B
\bbcal C
\bbcal D
\bbcal E
\bbcal F
\bbcal G
\bbcal H
\bbcal I
\bbcal J
\bbcal K
\bbcal L
\bbcal M
\bbcal N
\bbcal O
\bbcal P
\bbcal Q
\bbcal R
\bbcal S
\bbcal T
\bbcal U
\bbcal V
\bbcal W
\bbcal X
\bbcal Y
\bbcal Z

\deffrak b
\deffrak c
\deffrak C
\deffrak d
\deffrak h
\deffrak j
\deffrak k

\deffrak m
\deffrak n
\deffrak p
\deffrak q
\deffrak r
\deffrak s
\deffrak t
\deffrak u
\deffrak v

\def\al{\mat{\alpha}}
\def\be{\mat{\beta}}

\def\Ga{\mat{\Gamma}}

\def\eps{\mat{\varepsilon}}

\def\hi{\mat{\chi}}

\def\la{\mat{\lambda}}

\def\Om{\mat{\Omega}}
\def\om{\mat{\omega}}
\def\vi{\mat{\varphi}}
\def\si{\mat{\sigma}}

\def\te{\mat{\theta}}

\redefmath{eta}
\redefmath{xi}
\redefmath{mu}
\redefmath{nu}
\redefmath{Phi}
\redefmath{Psi}
\redefmath{psi}
\redefmath{pi}
\redefmath{rho}

\def\mrm@#1{\mat{\mathrm{#1}}}

\def\opr#1{\operatorname{#1}}
\def\DMO#1#2{\defmath{#1}{\opr{#2}}}
\def\oper#1{\defmath{#1}{\opr{#1}}}

\oper{asd}

\DMO{Hom}{Hom}
\DMO{RHom}{RHom}
\DMO{lHom}{\lH\mathit{om}}
\DMO{Ext}{Ext}
\DMO{lExt}{\lE\mathit{xt}}
\DMO{End}{End}
\DMO{Aut}{Aut}
\DMO{Fun}{Fun}
\DMO{Tor}{Tor}
\DMO{ext}{ext}

\DMO{Ob}{Ob}
\DMO{Mor}{Mor}
\DMO{im}{im}
\DMO{coim}{coim}
\DMO{coker}{coker}
\DMO{Arr}{Arr}

\DMO{Id}{Id}
\DMO{id}{id}
\DMO{add}{add} 
\DMO{ind}{ind} 
\DMO{pro}{pro} 
\DMO{Map}{Map} %
\DMO{Iso}{Iso} %
\DMO{Isom}{Isom}%
\DMO{Ind}{Ind}
\DMO{Bun}{Bun}
\DMO{Higgs}{Higgs}
\DMO{Hitch}{Hitch}

\DMO{Loc}{Loc}
\DMO{Maps}{Maps}
\makeatletter
\@ifpackageloaded{bbm}
{}
{}
\makeatother

\DMO{Presh}{Presh}

\DMO{coalg}{Coalg}
\DMO{Rep}{Rep}
\DMO{Cor}{Cor}

\DMO{Mod}{Mod}
\DMO{rad}{rad}
\DMO{soc}{soc}
\DMO{ann}{ann}
\DMO{pd}{pd}

\DMO{Spec}{Spec}
\DMO{Specm}{Specm}
\DMO{Max}{Max}
\DMO{spec}{Spec}
\DMO{Proj}{Proj}
\DMO{supp}{supp}
\DMO{Coh}{Coh}
\DMO{coh}{coh}
\DMO{Qcoh}{QCoh}
\DMO{QCoh}{QCoh}
\DMO{Pic}{Pic}
\DMO{Div}{Div}
\DMO{ch}{ch}
\DMO{Hilb}{Hilb}
\DMO{Fitt}{Fitt}
\DMO{Quot}{Quot}

\DMO{Gras}{Gr}
\DMO{Grass}{Gr}
\DMO{Flag}{Flag}
\DMO{Jac}{Jac}

\DMO{cone}{cone}
\DMO{Tw}{Tw}
\DMO{rank}{rk}
\DMO{rk}{rk}
\DMO{codim}{codim}
\DMO{cov}{cov}
\DMO{sgn}{sgn}
\DMO{td}{td}
\DMO{GL}{GL}
\DMO{PGL}{PGL}
\DMO{SL}{SL}
\DMO{SO}{SO}

\DMO{Der}{Der}
\DMO{der}{Der}
\DMO{coder}{Coder}
\DMO{diag}{diag}
\DMO{HMod}{HMod} 
\DMO{ad}{ad}
\DMO{Ad}{Ad}

\DMO{Ho}{Ho}

\DMO{har}{char}
\DMO{sk}{sk}
\DMO{cosk}{cosk}
\DMO{Gal}{Gal}
\DMO{tr}{tr}
\DMO{Tr}{Tr}
\DMO{Sh}{Sh}
\DMO{Is}{Is} 
\DMO{Hol}{Hol} 
\DMO{Lie}{Lie} 
\DMO{Res}{Res} 
\DMO{irr}{irr} %
\DMO{Irr}{Irr} %
\DMO{Exp}{Exp} %
\DMO{Log}{Log} %
\DMO{Pow}{Pow}
\DMO{pow}{pow}
\DMO{mult}{mult} %
\DMO{height}{ht} %
\DMO{wt}{wt}
\DMO{Vect}{Vect}
\DMO{moda}{mod}
\DMO{hd}{hd} 
\DMO{face}{face}
\DMO{Sym}{Sym}
\DMO{Com}{Com} 
\DMO{Eig}{Eig} 
\DMO{cl}{cl} 
\DMO{Li}{Li} 
\DMO{Imxx}{Im} 
\DMO{Rexx}{Re}

\DMO{St}{St}
\DMO{Var}{Var}

\def\norm#1{\mat{\lVert#1\rVert}}
\def\n#1{\mat{\lvert#1\rvert}}

\def\dd{\mat{\partial}}

\def\ts{\otimes}
\def\tl#1{\mat{\tilde{#1}}}

\def\what#1{\mat{\widehat{#1}}}

\def\sb{\subset}

\def\imp{\mat{\Rightarrow}}

\def\xx{\times}

\def\half{\mat{\frac12}}
\def\oh{\half} 

\def\inv{^{-1}}
\def\dual{^\vee}

\def\sst{{\rm{sst}}}

\def\ang#1{\mat{\left\langle #1\right\rangle}}

\def\set#1{\mat{\{ #1\}}}
\def\sets#1#2{\mat{\{ #1 \mid #2\}}}

\def\emb{\hookrightarrow}
\def\mto{\mapsto}
\def\arr{\futurelet\test\arrtest}
\def\arrtest{\ifx^\test\let\next\arra\else\let\next\arrb\fi\next}
\def\arra^#1{\xrightarrow{#1}} \def\arrb{\to}

\def\arrowsD{
\def\mto{{\:\vrule height .9ex depth -.2ex width .04em\!\!\!\;\ar}}
\def\ar{{\:\vrule depth -.52ex height .60ex width 0.85em\;\!\!\rhla\,}}
\def\arr{\futurelet\test\arrtest}
\def\arrtest{\ifx^\test\let\next\arra\else\let\next\arrb\fi\next}
\def\arra^##1{\rTo^{##1}} \def\arrb{\ar}
\def\emb{\futurelet\test\embtest}
\def\embtest{\ifx^\test\let\next\emba\else\let\next\embb\fi\next}
\def\emba^##1{\rInto^{##1}} \def\embb{{\:\rthooka\!\!\!\ar}}
\newarrow{Eq}=====
\def\rrarr{\pile{\rTo\\ \rTo}}
\def\lrarr{\pile{\rTo\\ \lTo}}   
\newarrow{ShortTo}{}{}-->
}

\def\arrowsDStandard{
\newarrow{TeXto}----{->}
\newarrow{TeXinto}C---{->}
\newarrow{TeXonto}----{->>}
\newarrow{TeXdashto}{}{dash}{}{dash}{->}
\newarrow{Eq}=====
\def\ar{\rightarrow}
\def\emb{\futurelet\test\embtest}
\def\embtest{\ifx^\test\let\next\emba\else\let\next\embb\fi\next}
\def\emba^##1{\rTeXinto^{##1}} \def\embb{\hookrightarrow}
\def\rrarr{\pile{\rTo\\ \rTo}}
\def\lrarr{\pile{\rTo\\ \lTo}}   
}

\def\theorems{
\newcounter{nthr} 
\numberwithin{nthr}{section}
\let\theHnthr\thenthr
\newtheorem{thr}[nthr]{Theorem}
\newtheorem{prp}[nthr]{Proposition}
\newtheorem{lmm}[nthr]{Lemma}
\newtheorem{crl}[nthr]{Corollary}
\newtheorem{clm}[nthr]{Claim}
\newtheorem{conj}{Conjecture}
\newtheorem{rmr}[nthr]{Remark}
\theoremstyle{definition}
\newtheorem{dfn}[nthr]{Definition}
\newtheorem{exm}[nthr]{Example}
\newtheorem{claim}[nthr]{Claim}
\theoremstyle{remark}
}

\newif\ifrem\remtrue
\def\rem#1{\ifrem\marginpar{\raggedright\footnotesize #1}\fi}

\def\lb#1{\mat{\underline{#1}}} 
\def\ub#1{\mat{\overline{#1}}}  
\def\over#1#2{\mat{\substack{#1\\#2}}} 

\def\ie{i.e.\ }
\def\eg{e.g.\ }
\def\cf{cf.\ }

\def\eprint#1#2{%
\expandafter\ifx\csname eprint@#1\endcsname\relax#1:#2%
\else\def\itemID{#2}\csname eprint@#1\endcsname\fi}

\def\defArchive#1#2{%
\makedef{eprint@#1}{#2}}

\defArchive{arxiv}{\href{http://arXiv.org/abs/\itemID}{{\tt arXiv:\itemID}}}
\defArchive{numdam}{\href{http://www.numdam.org/item?id=\itemID}{Numdam:\itemID}}

\theorems\arrowsDStandard

\begin{document}
\remfalse
\title[Wall-crossing formulas]{Wall-crossing formulas for framed objects}
\author{Sergey Mozgovoy}%
\email{mozgovoy@maths.ox.ac.uk}%
\begin{abstract}
We prove wall-crossing formulas for the motivic invariants of the moduli spaces of framed objects in the ind-constructible abelian categories. Developed techniques are applied in the case of the motivic Donaldson-Thomas invariants of quivers with potentials. Another application is a new proof of the formula for the motivic invariants of smooth models of quiver moduli spaces.
\end{abstract}
\maketitle
\tableofcontents

\section{Introduction}
The goal of this paper is to prove general wall-crossing formulas for the motivic invariants of the moduli spaces of framed objects in the ind-constructible abelian categories.
Some of the formulas of this type were previously discovered by various authors in various contexts \cite{chuang_chamber,engel_smooth,nagao_counting,toda_curve}. Every time the tools to prove these results were developed from scratch.
In this paper we will develop a general framework for the wall-crossing formulas for framed objects in the abelian categories. Instead of giving the definition of an abelian category with framing right now, let me give a couple of examples. 

Let $X$ be an algebraic variety and let \lA be an abelian category consisting of morphisms $f:V\ts \lO_X\to F$, where $V$ is a vector space and $F$ is a coherent sheaf over $X$ (sometimes one considers a fixed coherent sheaf instead of $\lO_X$ \cite{huybrechts_framed}). Morphisms between the objects $f:V\ts \lO_X\to F$ and $g:W\ts \lO_X\to G$ in \lA are pairs of morphisms $(V\to W,F\to G)$ such that the corresponding square diagram commutes. We define the framing rank $v:K_0(\lA)\to\cZ$ by $v(f)=\dim V$. Note that the category $\lA^0\sb \lA$ of unframed objects, consisting of objects $f\in\lA$ such that $v(f)=0$, is just the category of coherent sheaves over $X$. Framed sheaves were studied for example in \cite{huybrechts_framed,pandharipande_curve,thaddeus_stable}. If $X$ is a smooth $3$-Calabi-Yau manifold with $H^1(X,\lO_X)=0$, let \lA be the category consisting of morphisms $V\ts \lO_X\to F$ as above, with $F\in\Coh_{\le1}(X)$ (\ie having support of dimension $\le1$). Category \lA is a full abelian subcategory of an abelian subcategory
$$\lA_X=\ang{\lO_X,\Coh_{\le1}(X)[-1]}_{\rm ex}$$
of $D^b(\Coh (X))$ studied by Toda \cite{toda_curve}. We can define the framing rank also on $K_0(\lA_X)$ by $F\mto\rk F$. The subcategories of \lA and $\lA_X$ consisting of objects having framing rank $0$ are both equivalent to $\Coh_{\le1}(X)$. For more details see Example \ref{exm:toda}.
\rem{if $X$ is simply connected then $H^1(X,\lO_X)=0$. we need this}

Let $(Q,W)$ be a quiver with potential and let $J_{Q,W}=\cC Q/(\dd W)$ be the corresponding Jacobian algebra. Given $w\in\cN^{Q_0}$, we construct a new quiver $Q'$ by adding a new vertex $*$ to $Q_0$ and adding $w_i$ arrows from $*$ to $i$ for every $i\in Q_0$. Let $J_{Q',W}=\cC Q'/(\dd W)$ be the Jacobian algebra of $Q'$ with respect to the potential $W$. Let $\lA=\moda J_{Q',W}$ be the category of finite-dimensional representations of $J_{Q',W'}$. We define the framing rank $v:K_0(\lA)\to\cZ$ by $v(M)=\dim M_*$, where $M_*$ is a vector space concentrated at vertex $*\in Q'_0$. Note that the category $\lA^0\sb \lA$ consisting of objects $M$ with $v(M)=0$ is just the category of representations of $J_{Q,W}$. Framed representations for quivers with trivial potentials were studied in \cite{engel_smooth,reineke_cohomology}. Framed representations for quivers with potentials arising from brane tilings were studied in \cite{behrend_motivic,mozgovoy_noncommutative,nagao_counting,szendroi_non-commutative} and many other works.

Assume now that we have an abelian category \lA with a framing rank $v:K_0(\lA)\to\cZ$ and a stability function $Z:K_0(\lA)\to\cC$. Under some conditions the functions $Z_c=Z-cv:K_0(\lA)\to\cC$, $c\in\cR$, are again stability functions. We will study how the property of $Z_c$-stability for objects $E\in\lA$ with framing rank $v(E)=1$ changes when we slightly shift $c$. The most strong results in this direction were proved by Diaconescu \cite{diaconescu_chamber} in the case of ADHM sheaves (they form a framed abelian category). Related results were proved in \cite{nagao_counting} in the case of the quiver with potential arising from the conifold. Our results relate $Z_c$-semistable objects (having framing rank one) with $Z_{c\pm\eps}$-stable objects (also having framing rank one) for $0<\eps\ll1$. This relation can be translated into the wall-crossing formula (see Theorem \ref{hall wall cross}) in the Hall algebra of \lA. It envolves the generating functions for $Z_c$- and $Z_{c\pm\eps}$-semistable objects with framing rank one and the generating function for $Z$-semistable unframed objects.

The wall-crossing formula in the Hall algebra can be translated into the wall-crossing formula in the quantum torus (determined by a skew-symmetric form on $K_0(\lA)$) after applying the integration map. There is a very powerful construction of the integration map due to Kontsevich and Soibelman \cite{kontsevich_stability} for an arbitrary ind-constructible $3$-Calabi-Yau category with some additional data. In this paper, however, we will consider two simpler constructions of the integration maps. The first one is the integration map of Reineke \cite{reineke_harder-narasimhan,reineke_counting} for quivers with trivial potentials (its motivic version was constructed by Joyce \cite{joyce_configurationsa}). The second one is the integration map for quivers with potentials constructed in \cite{mozgovoy_motivica}. Accordingly, we will give two applications of the developed techniques. Using the first integration map we will give a new proof of the formula for the invariants of smooth models of quiver moduli spaces \cite{engel_smooth}. Using the second integration map we will derive some interesting information on the non-commutative Donaldson-Thomas invariants and some relations between the motivic DT invariants for framed representations and
the motivic DT invariants for unframed representations. We will especially emphasize the role of universal DT invariants for the computation of all other DT invariants (framed and unframed) for arbitrary stability parameters.

Especially simple wall-crossing formulas are obtained in the symmetric case (this means that the restriction of the skew-symmetric form to the unframed part is zero). In the case of a quiver with potential $(Q,W)$ this is equivalent to the condition that the quiver $Q$ is symmetric (\ie the number of arrows from $i$ to $j$ equals the number of arrows from $j$ to $i$ for any $i,j\in Q_0$). Note that we do not require the extended quiver $Q'$ to be symmetric. The other important symmetric case is provided by the category $\lA_X$ for a $3$-Calabi-Yau manifold $X$ described earlier. The reason is that the Euler characteristic $\hi(F,G)=0$ for any $F,G\in\Coh_{\le1}(X)$.
\rem{ (note that $\ch_0F=\ch_1F=0$ for $F\in\Coh_{\le1}(X)$ and apply the Hirzebruch-Riemann-Roch theorem). }

If a quiver with potential $(Q,W)$ arises from a consistent brane tiling then $Q$ is rarely symmetric (we say then that the brane tiling is symmetric). The only known examples arise from those brane tilings that have the toric diagram without interior lattice points (by \cite{aganagic_wall} these toric diagrams are trapezoids with height one and the toric diagram corresponding to $\cC^3/(\cZ_2\xx\cZ_2)$). It is natural to conjecture that these are the only symmetric brane tilings. The author was informed by Alastair King that a more strong assertion is true. Namely, for any consistent brane tiling the rank of the skew-symmetric form of the corresponding quiver equals to the number of interior lattice points of the corresponding toric diagram. Interestingly enough, only for symmetric brane tilings there are known explicit formulas for the classical Donaldson-Thomas invariants (also called generalized Donaldson-Thomas invariants \cite{joyce_theory}), see \eg \cite{gholampour_counting,nagao_deriveda,nagao_counting,szendroi_non-commutative,young_computing,young_generating}. The author expects that the motivic Donaldson-Thomas invariants can also be computed in all these cases.  

I would like to thank Tamas Hausel, Kentaro Nagao, and Bal\'azs Szendr\H oi for many useful discussions. 
The author's research was supported by the EPSRC grant EP/G027110/1.

\def\qt{\mat{\mathbb T}}
\def\cqt{\mat{\what{\mathbb T}}}
\def\Ieq{I_{\rm eq}}
\def\hall{\what H}
\def\eq{H_{\rm eq}}
\def\heq{\what H_{\rm eq}}
\def\mot{\opr{Mot}_{\rm st}}
\def\bb{B}
\oper{CM}

\section{Preliminaries}

\subsection{Ring of motives}
Let $\lM=K_0(\CM_\cC)$ be the Grothendieck ring of the category of effective Chow motives over \cC with rational coefficients. It is known that $\lM$ is a (special) \la-ring \cite{getzler_mixed,heinloth_note}. 

\begin{rmr}
The reason why we work with $K_0(\CM_\cC)$ instead of a more usual ring $K_0(\Var_\cC)$ of algebraic varieties is that $K_0(\CM_\cC)$ is a (special) \la-ring, while $K_0(\Var_\cC)$ is, as far as the author knows, only a pre-\la-ring. There is a natural ring homomorphism $h:K_0(\Var_\cC)\to K_0(\CM_\cC)$ that respects the pre-\la-ring structures. In particular $h([S^nX])=\si_n h([X])$ for any smooth projective variety $X$, where $\si_n$ are the \si-operations on the \la-ring $K_0(\CM_\cC)$, see \cite{banorollin_motive,gillet_descent,guill'en_un,heinloth_note}.
\end{rmr}

Let $\cL=[\cA^1]\in\lM$ be the Lefschetz motive. Let $\what\lM$ be the dimensional completion of $\lM_\cL=\lM[\cL\inv]$ defined as follows (\cf \cite{behrend_motivica}). Consider a decreasing group filtration $F^n\lM_\cL\sb\lM_\cL$, $n\ge0$, where $F^n\lM_\cL$ is generated by the elements $$[X]\cL^{-k},\qquad\dim X-k\le -n.$$
Every $F^n\lM_{\cL}$ is an ideal in the ring $F^0\lM_\cL$ and we can construct the corresponding ring completion $\what{F^0\lM_\cL}$. We define then
$$\what\lM=\what{F^0\lM_\cL}\ts_{F^0\lM_\cL}\lM_\cL.$$
Finally, we define $\lV=\what\lM[\cL^\oh]$. This ring still has a structure of a \la-ring, where we extend the Adams operations by
$$\psi_n(\cL^\oh)=\cL^{\frac n2}.$$
Note that the elements $1-\cL^n$ as well as $[\GL_n]$ are invertible in~\lV. 

\subsection{Quantum torus}
\label{qt}
Let $\Ga$ be a free finitely generated abelian group endowed with a skew symmetric bilinear form $\ang{-,-}$. Let $R$ be a commutative ring and let $q^\oh\in R$ be some invertible element. 

\begin{dfn}
Define the quantum torus $\qt_\Ga=\qt_{\Ga,R}$ to be the algebra isomorphic as a vector space to the group algebra $R[\Ga]$ (we denote its basis elements by $x^\al,\ \al\in\Ga$) and endowed with multiplication
$$x^\al\circ x^\be=(-q^\oh)^{\ang{\al,\be}}x^{\al+\be}\qquad \al,\be\in\Ga.$$ 
\end{dfn}

\begin{rmr}
If $R=\lV$ is the ring of motives, then we choose $q^\oh=\cL^\oh$. The corresponding quantum torus is called the motivic quantum torus.
\end{rmr}

Let $\Ga_+\sb\Ga$ be a pointed semigroup, \ie
$$\Ga_+\cap(-\Ga_+)=\set0.$$
We define a partial order on \Ga by declaring $\al\le\be$ if $\be-\al\in\Ga_+$. We will always assume that for any $\al\in\Ga_+$ the number of elements $0<\be<\al$ is finite. If $\Ga_+$ is finitely generated then our assumption is satisfied and this will be the case in all our examples. 

\begin{rmr}
\label{contra}
The semigroup $\Ga_+\sb\cZ^2$ given by
$$\Ga_+=(\cZ_{>0}\xx\cZ)\cup(\set0\xx\cZ_{\ge0})$$
is pointed but it does not satisfy our last assumption. Indeed, we have $0<(1,n)<(2,0)$ for any $n\in\cZ$. This semigroup is not finitely generated.
\end{rmr}

\begin{rmr}
Another family of semigroups satisfying our assumption comes from \cite[Section 2.2]{kontsevich_stability}. Let $Q$ be a quadratic form on \Ga and let $Z:\Ga\to\cC$ be a group homomorphism such that $Q|_{\ker Z}<0$. Let $V\sb\cC$ be a strict sector (that is a convex cone such that its closure is pointed, that is $\ub V\cap(-\ub V)=\set0$). Define the semigroup $\Ga_+\sb\Ga$ to be generated by the elements
\begin{equation}
\sets{\al\in\Ga}{Z(\al)\in V,Q(\al)\ge0}.
\label{generating}
\end{equation}
Then for any $\al\in\Ga_+$ the number of elements $0<\be<\al$ is finite. For simplicity, let us assume that $Q$ has signature $(2,n-2)$. Then we can write $Q$ in the form $Q(\al)=-\norm\al+C\n{Z(\al)}$, for some norm $\norm-$ on \Ga and some $C>0$. Assume that there exists $\al\in\Ga_+$ with an infinite number of elements $0<\be<\al$. The semigroup $S=Z(\Ga)\cap\ub V$ is finitely generated, therefore there exists a finite number of element $z\in S$ such that $Z(\al)-z\in S$. This implies that there exists an infinite number of elements $0<\be<\al$ with the same value $Z(\be)=z$. We can assume that these $\be$ are elements of the generating set \eqref{generating}. Therefore they satisfy $\norm\be\le C\n{Z(\be)}=C\n z$. But there exists only a finite number of element in \Ga inside a fixed sphere.

It follows from the example in Remark \ref{contra} that the sector
$$V=(\cR_{>0}\xx\cR)\cup(\set0\xx\cR_{\ge0})$$
will not work also it does not contain a straight line (but its closure does).
\end{rmr}

\begin{dfn}
\label{completion}
Consider the subalgebra $\qt_{\Ga_+}:=\oplus_{\al\in\Ga_+}Rx^\al\sb\qt_{\Ga}$.
We define its completion to be the algebra $\cqt_{\Ga_+}:=\prod_{\al\in\Ga_+}Rx^\al$
with the same product rule as above. It follows from our assumption that multiplication is well defined. Define the complete quantum torus to be the algebra
$$\cqt_{\Ga_+,\Ga}:=\cqt_{\Ga_+}\ts_{\qt_{\Ga_+}}\qt_{\Ga}.$$
We call it the completion of $\qt_\Ga$ with respect to $\Ga_+$.
\end{dfn}

If $R$ is a \la-ring, then we endow also $\qt_\Ga$ (and $\cqt_{\Ga_+},\cqt_{\Ga_+,\Ga}$) with a structure of a \la-ring by the rule
$$\psi_n(ax^\al)=\psi_n(a)x^{n\al},\qquad a\in R,\al\in\Ga.$$
Using this \la-ring structure we can endow the complete algebra $\cqt_{\Ga_+}$ with a plethystic exponential \cite{getzler_mixed,mozgovoy_computational}
$$\Exp(f)=\sum_{n\ge0}\si_n(f)=\exp\left(\sum_{n\ge1}\frac1n\psi_n(f)\right).$$
 

\subsection{Stability functions}
For a comprehensive introduction to stability functions see \cite{bridgeland_stability}.
Let $\lA$ be an abelian category and let $K_0(\lA)$ be its Grothendieck group. We fix once and for all a group homomorphism
$$\cl:K_0(\lA)\to\Ga,$$
where $\Ga$ is a free finitely generated abelian group, endowed with a skew-symmetric bilinear form $\ang{-,-}$.

\begin{rmr}
Usually one requires that for any $E,F\in\lA$ the number $\ang{\cl E,\cl F}$ is somehow related to the dimensions of the $\Ext$-groups $\Ext^i(E,F)$, $\Ext^i(F,E)$ for $i\ge0$. But we will not require this condition.
\end{rmr}

\begin{dfn}
A stability function (or a central charge) on a category \lA is a group homomorphism
$Z:\Ga\to\cC$ such that, for every nonzero $E\in\lA$ (we write $Z(E)$ for $Z(\cl E)$),
$$Z(E)\in \cH=\sets{re^{i\pi\vi}}{r>0,\vi\in(0,1]}\sb\cC.$$
For any nonzero $E\in\lA$ there exist uniquely determined real numbers $m(E)>0$, $\vi(E)\in(0,1]$ ($\vi(E)$ is called the phase of $E$) such that
$$Z(E)=m(E)e^{i\pi\vi(E)}.$$
\end{dfn}

\begin{dfn}
An object $E\sb\lA$ is called $Z$-semistable (resp.\ $Z$-stable) if for any $0\ne F\subsetneq E$ we have $\vi(F)\le\vi(E)$ (resp.\ $\vi(F)<\vi(E)$).
\end{dfn}

\begin{dfn}
A stability function $Z$ is said to have the Harder-Narasimhan property if for any object $E$ there is a finite chain of subobjects, called a HN filtration
$$0=E_0\sb\dots\sb E_n=E$$
such that $E_i/E_{i-1}$, $i=1\dots n$, are semistable and
$$\vi(E_1/E_0)>\dots>\vi(E_n/E_{n-1}).$$
\end{dfn}
One can show that HN filtrations are unique.

\begin{rmr}
There exist group homomorphisms $r,d:\Ga\to\cR$ such that
$$Z=-d+ir.$$
One can show that for any object $0\ne E\in\lA$, $r(E)\ge0$ and if $r(E)=0$ then $d(E)>0$. Conversely, given such group homomorphisms $r,d:\Ga\to\cR$, the function $Z:\Ga\to\cC$ defined by $Z=-d+ir$ is a stability function. We define the slope function
$$\mu(E)=\mu_Z(E)=\frac{d(E)}{r(E)}=-\cot(\pi\vi(E))\in(-\infty,\infty].$$
Then one can easily see that $\mu(E)\le\mu(F)$ if and only if $\vi(E)\le\vi(F)$ for $E,F\in\lA$. Therefore we can use the slope function and the phase function in the definition of stable and semistable objects interchangeably.
\end{rmr}

Let $\Ga_+\sb\Ga$ be a semigroup generated by the elements $\cl E$, where $E\in\lA$. It follows from the existence of a stability function on \Ga that $\Ga_+$ is pointed \ie
$$\Ga_+\cap(-\Ga_+)=\set0.$$
Therefore we can define a partial order on $\Ga$ as in Section \ref{qt}.
If otherwise not stated, we will assume that for any $\al\in\Ga_+$ the number of elements $0<\be<\al$ is finite.

\begin{dfn}
\label{generic}
Given $\al\in\Ga_+$, we say that the stability function $Z$ is \al-generic if for any $0<\be<\al$ the rays $\cR_{>0} Z(\be)$ and $\cR_{>0} Z(\al)$ are not equal (equivalently, $\vi(\be)\ne\vi(\al)$).
\end{dfn}

If $E\in\lA$ is $Z$-semistable and $Z$ is $\cl E$-generic then $E$ is automatically stable.

\subsection{Motivic Hall algebra}
Let \lA be an abelian ind-constructible category over \cC \cite[Definition 8]{kontsevich_stability}. In particular, the set of objects $\Ob\lA$ is an ind-constructible set over \cC (we can always assume that it is a countable disjoint union $\sqcup_{i\in I} Y_i$ of varieties over \cC, in particular it is a scheme locally of finite type over \cC) and the groupoid of isomorphisms $\lM=\Iso\lA$ is an ind-constructible stack with affine stabilizers (we can always assume that it is a countable disjoint union of global quotients $\sqcup_{i\in I} [Y_i/\GL_{n_i}]$, in particular it is an algebraic stack locally of finite type over \cC with affine stabilizers). 

The importance of stacks of finite type over \cC with affine stabilizers is justified by the fact that one can define their motives \cite{behrend_motivica,bridgeland_introduction}.
For an algebraic stack $S$ locally of finite type over \cC with affine stabilizers one defines the group $\mot(S)$ (see \cite[Section 4.2]{kontsevich_stability} or \cite[Section 3.4]{bridgeland_introduction}, where it is denoted by $K(\St/S)$ or \cite{joyce_configurationsa}, where it is denoted by $\lb {\opr{SF}}(S)$) to be the group generated by the isomorphism classes of $1$-morphisms of stacks $X\to S$, where $X$ is an algebraic stack of finite type type over \cC (not just locally of finite type) with affine stabilizers, modulo standard relations, loc.\ cit. In the same way one can define the group $\mot(S)$ if $S$ is an ind-constructible stack with affine stabilizers.
One defines the Hall algebra $H(\lA)$ to be $\mot(\lM)$ endowed with multiplication that mimics the usual Ringel multiplication, loc.\ cit. 

Under the assumptions of the previous section we can write $\lM=\sqcup_{\al\in\Ga_+}\lM_\al$, where $\lM_\al$ consists of objects having class \al. The algebra $H(\lA)$ is $\Ga_+$-graded. We define the complete Hall algebra $\hall=\hall(\lA)$ to be
$$\hall(\lA)=\prod_{\al\in\Ga_+}H(\lA)_\al.$$

The following result is a version of \cite[Lemma 3.3]{reineke_counting} (see also \cite[Theorem 6.4]{joyce_configurationsa}).

\begin{prp}
\label{reineke map}
Let \lA be an ind-constructible abelian category of homological dimension $1$. Assume that there exists a bilinear form $\hi$ on \Ga such that for any $M,N\in\Ob\lA$ we have
$$\hi(\cl M,\cl N)=\dim \Hom(M,N)-\dim\Ext^1(M,N)$$
and $\ang{\al,\be}=\hi(\al,\be)-\hi(\be,\al)$ for $\al,\be\in\Ga$.
Then the map
$$I:\hall\to\cqt_{\Ga_+},\qquad [X\to\lM_\al]\mto (-\cL^\oh)^{\hi(\al,\al)}[X]x^\al$$
is an algebra homomorphism, where $[X\to\lM_\al]$ is an element in $\mot(\lM_\al)$ and $[X]\in\lV$ is the motive of the stack $X$ having affine stabilizers \cite{behrend_motivica,bridgeland_introduction}. We call the map $I$ the integration map.
\end{prp}

\begin{rmr}
There are some differences of the stated result from the result proved in \cite{reineke_counting}.
First of all the Hall algebra used here (after \cite{bridgeland_introduction,joyce_configurationsa,kontsevich_stability}) is opposite to the usual Ringel-Hall algebra used in \cite{reineke_counting}. Next, our multiplication in the quantum torus is slightly different from the multiplication in the quantum torus used in \cite{reineke_counting}, but these two algebras are canonically isomorphic and our integration map incorporates the isomorphism. And finally, \cite[Lemma 3.3]{reineke_counting} is proved for the Hall algebra of the category of quiver representations over a finite field $\cF_q$, with an integration map given by (if we use our conventions on the multiplication in the Hall algebra and the quantum torus)
$$[M]\mto\frac{(-q^\oh)^{\hi(\al,\al)}}{\#\Aut M}x^{\al},\qquad \al=\lb\dim M$$
for any quiver representation $M$.
\end{rmr}

\section{Stability for framed objects}
\begin{dfn}
A group homomorphism $v:\Ga\to\cZ$ is called a framing rank if for any $E\in\lA$ we have $v(E)\ge0$. Stability $Z=-d+ir$ is called $v$-compatible if whenever $r(E)=0$, we have also $v(E)=0$. 
\end{dfn}

Let $v$ be a framing rank and let $Z=-d+ir$ be a $v$-compatible stability function.
For any $c\in\cR$, we define a new stability function $Z_c=Z-cv$. The corresponding slope function is given by
$$\mu_c(E)=\mu_{Z_c}(E)=\frac{d(E)+cv(E)}{r(E)}.$$
We assume that for any $c\in\cR$ the stability function $Z_c$ has the HN property. 

\begin{dfn}
We say that an object is $c$-semistable (resp.\ $c$-stable) if it is semistable (resp.\ stable) with respect to $Z_c$. We say that $c\in\cR$ is \al-generic if $Z_c$ is \al-generic (see Definition~\ref{generic}).
\end{dfn}

Our goal is to study $c$-semistable objects with $v=1$
and to study their behavior when we slightly shift $c$. Many results of this section can be found in \cite{diaconescu_chamber} in a slightly different form.

\begin{rmr}
Our condition on $Z$-compatibility of the framing rank $v$ means that $r(E)$ incorporates both the rank of the unframed object and the framing rank. This is slightly different from the approach of \cite{diaconescu_chamber}, where the deformed slope function did not actually came from the stability function and one had to consider the objects with $Z(E)=0$ and $v(E)>0$ separately from the other cases.
\end{rmr}

\begin{lmm}
\label{unique}
Let $0<\be<\al$ and $v(\al)=1$. Then there exists at most one $c\in\cR$ such that $\mu_c(\be)=\mu_c(\al)$.
\end{lmm}
\begin{proof}
Let $Z(\al)=-d+ir$, $Z(\be)=-d'+ir'$, $v(\al)=v$, $v(\be)=v'$. Assume that there exist $c\ge0$, $s\ge0$ such that
$$d'+cv'=s(d+cv),\qquad r'=sr.$$
Then $s=r'/r\le1$ is uniquely determined and we have
$$c(v'-vr'/r)=d'-dr'/r.$$
If $v'-vr'/r\ne0$ then $c$ is uniquely determined. Assume that $v'r=vr'$. Then also $d'r=dr'$. If $v'=1$ then $r'=r$ and therefore $d'=d$. This contradicts our assumption $\be<\al$. If $v'=0$ then also $r'=0$ and therefore $d'=0$. This contradicts our assumption $\be>0$.
\end{proof}

\begin{crl}
Let $\al\in\Ga_+$ be such that  $v(\al)=1$. Then there exist only a finite number of $c\in\cR$ which are not \al-generic.
\end{crl}
\begin{proof}
We apply the previous lemma together with an assumption that there exists just a finite number of elements $0<\be<\al$. 
\end{proof}

\begin{dfn}
Let $E\in\lA$ be an object with $v(E)=1$ and let $c\in\cR$. We say that $E$ is $c^+$-stable (or $Z_{c^+}$-stable) if it is stable with respect to $c^+=c+\eps$ for $0<\eps\ll1$. In the same way we define $c^-$-stable (or $Z_{c^-}$-stable) objects.
We say that $E$ is $+\infty$-stable (or $Z_{+\infty}$-stable) if $E$ is stable with respect to $c\gg0$. In the same way we define $-\infty$-stable (or $Z_{-\infty}$-stable) objects.
\end{dfn}

\begin{rmr}
If $\al\in\Ga_+$ is such that $v(\al)=1$, then $c^+$ and $c^-$ are automatically \al-generic. Therefore the notions of $c^\pm$-stability and $c^\pm$-semistability coincide for the objects having class \al.
\end{rmr}

\begin{lmm}
\label{basic properties}
Let $E\in\lA$ be an object with $v(E)=1$.
\begin{enumerate}
	\item If $E$ is $c$-stable then it is $c^\pm$-stable.
	\item If $E$ is $c^+$-stable or $c^-$-stable then it is $c$-semistable.
	\item 
If $E$ is $c^+$-stable and $c^-$-stable then it is $c$-stable.
\end{enumerate}
\end{lmm}
\begin{proof}
$1$.
Assume that $E$ is $c$-stable and let $F\sb E$ be its $c^-$-destabilizing subobject. Then
$$\mu_c(F)\ge\mu_{c^-}(F)\ge \mu_{c^-}(E).$$
Taking the limit we get $\mu_c(F)\ge\mu_{c}(E)$
that contradicts to the $c$-stability of $E$.
If $F\sb E$ is $c^+$-destabilizing subobject then
$$\mu_c(E/F)\le\mu_{c^+}(E/F)\le \mu_{c^+}(E).$$
Taking the limit we get $\mu_c(E/F)\le\mu_{c}(E)$
that contradicts to the $c$-stability of $E$.

$2$. We just have to take the limit.

$3$. Assume there exists a $c$-destabilizing subobject $F\sb E$.
If $v(F)=0$ then
$$\mu_{c^-}(F)=\mu_c(F)\ge\mu_c(E)>\mu_{c^-}(E)$$
contradicting to $c^-$-stability of $E$.
If $v(1)=0$ then $v(E/F)=0$ and
$$\mu_{c^+}(E/F)=\mu_c(E/F)\le\mu_c(E)<\mu_{c^+}(E)$$
contradicting to $c^+$-stability of $E$.
\end{proof}

\begin{prp}
\label{main relation}
Let $E\in\lA$ be an object with $v(E)=1$. 
The following conditions are equivalent
\begin{enumerate}
	\item $E$ is $c$-semistable.
	\item The HN filtration of $E$ with respect to $c^-$ has the form
		$E'\sb E$ such that $v(E')=0$	and $E',E$ have equal $c$-slopes.
	\item The HN filtration of $E$ with respect to $c^+$ has the form
		$E'\sb E$ such that $v(E')=1$ and $E',E$ have equal $c$-slopes.
\end{enumerate}
\end{prp}
\begin{proof}
$1\imp2$. 
Let $E'\sb E$ be a $c^-$-destabilizing subobject. Let
$$Z(E)=-d+ir,\qquad Z(E')=-d'+ir'.$$
Then
$$(d'+c^-v')r>(d+c^-v)r',\qquad (d'+cv')r\le(d+cv)r'.$$
Taking the limits we deduce that
$$(d'+cv')r=(d+cv)r'.$$
Theefore $\mu_c(E')=\mu_c(E)$.
We also obtain
$$v'r< vr'.$$
This implies $v'=0$ and therefore 
$$\mu_{c^-}(E')=\mu_c(E')=\mu_c(E).$$
All this implies that there is a 2-step HN filtration $E'\sb E$ with respect to $c^-$
(or 1-step, if $E$ is already $c^-$-stable) and that $v(E')=0$.

$2\imp1$.
The objects $E',E/E'$ are $c^-$-semistable and therefore $c$-semistable by Lemma \ref{basic properties} (the notions of $c^-$-semistability and $c$-semistability coincide for $E'$ as $v(E')=0$).
They have the same $c$-slope, so $E$ is also $c$-semistable.

Equivalence of $1$ and $3$ is proved in the same way.
\end{proof}

\begin{rmr}
The above result says in particular that if $E$ is $c$-semistable with $v(E)=1$
then there exists the unique quotient $E\to E''$ such that $v(E'')=1$, $\mu_c(E'')=\mu_c(E)$, and $E''$ is
$c^-$-stable. Also
there exists the unique subobject $E'\sb E$ such that $v(E')=1$, $\mu_c(E')=\mu_c(E)$, and $E'$ is
$c^+$-stable.
\end{rmr}

%
%

\begin{prp}
Let $E\in\lA$ be a $c$-semistable object with $v(E)=1$.
\begin{enumerate}
	\item Let $E'\sb E$ be a HN filtration with respect to $c^-$. If $E$ is $c^+$-stable then	$E/E'$ is $c$-stable.
	\item Let $E'\sb E$ be a HN filtration with respect to $c^+$. If $E$ is $c^-$-stable then $E'$ is $c$-stable.
\end{enumerate}
\end{prp}
\begin{proof}
Let $E'\sb E$ be a HN filtration with respect to $c^-$. Assume that $E$ is $c^+$-stable.
The object $E''=E/E'$ is $c^-$-semistable and therefore also $c^-$-stable.
Let $F\sb E''$ be a $c$-destabilizing subobject and let $G=E''/F$.

If $v(F)=0$ then
$$\mu_{c^-}(F)=\mu_{c}(F)\ge\mu_{c}(E'')>\mu_{c^-}(E'')$$
that contradicts to $c^-$-stability of $E''$.
If $v(F)=1$ then $v(G)=0$ and
$$\mu_{c^+}(E)>\mu_{c}(E)=\mu_c(E'')\ge\mu_c(G)=\mu_{c^+}(G)$$
that contradicts to $c^+$-stability of $E$.

\end{proof}

\begin{rmr}
The converse statements of the proposition seem to be false.
The proposition generalizes \cite[Prop.~3.8]{nagao_counting}, where the authors suppose that the category of semistable objects $E$ with a fixed
slope and $v(E)=0$ is semisimple with one simple object.
\end{rmr}

\begin{prp}
\label{infty stability}
Let $Z=-d+ir$ be a stability function such that $r(\al)>0$ for any $0\ne\al\in\Ga_+$. Let $E\in\lA$ be an object with $v(E)=1$. Then
\begin{enumerate}
	\item $E$ is $+\infty$-stable if and only if for any subobject $E'\sb E$ we have $v(E')=0$.
	\item $E$ is $-\infty$-stable if and only if for any subobject $E'\sb E$ we have $v(E')=1$.
\end{enumerate}
\end{prp}
\begin{proof}
Let $0<\be<\al$ and let $Z(\al)=-d+ir$, $Z(\be)=-d'+ir'$, $v(\al)=v$, $v(\be)=v'$. Then $\mu_c(\be)<\mu_c(\al)$ means
$$\frac{d'+cv'}{r'}<\frac{d+cv}{r}$$
or equivalently
$$c(v'r-vr')<dr'-d'r.$$
This holds for $c\gg0$ if and only if either $v'r<vr'$ or $v'r=vr'$ and $d'r<dr'$.
In the second case if $v'=0$ then $r'=0$ and therefore $\be=0$. And if $v'=1$ then $r'=r$ and $\be=\al$. Both these cases are impossible. Therefore we have $v'r<vr'$.
This is true if and only if $v'=0$.

The above inequality holds for $c\ll0$ if and only if either $v'r>vr'$ or $v'r=vr'$ and $d'r<dr'$. The second case is again impossible. Therefore $v'r>vr'$, or equivalently $v'=1$.
\end{proof}

In the following example we will see that the requirement in the previous proposition that $r(\al)>0$ for any $0\ne\al\in\Ga_+$ is important.

\begin{exm}
\label{exm:toda}
Let $X$ be a projective $3$-Calabi-Yau manifold with $H^1(X,\lO_X)=0$. Following \cite{toda_curve} we define
$$\lD_X=\ang{\lO_X,\Coh_{\le1}(X)}_{\rm tr}\sb D^b(\Coh(X))$$
and
$$\lA_X=\ang{\lO_X,\Coh_{\le1}(X)[-1]}_{\rm ex}\sb\lD_X.$$
The category $\lA_X$ is the heart of a bounded t-structure on $\lD_X$ \cite[Lemma 3.5]{toda_curve}. The category $\lA_X$ contains a full abelian subcategory $\lA$ of morphisms $f:V\ts \lO_X\to F$, where $V$ is a finite-dimensional vector space and $F\in\Coh_{\le1}(X)$.
Define the map
$$\cl:K_0(\lD_X)=K_0(\lA_X)\to \Ga=H^0(X,\cZ)\oplus H^4(X,\cZ)\oplus H^6(X,\cZ)$$
by the rule $E\mto(\ch_0E,\ch_2E,\ch_3E)$ (note that for any $E\in\lD_X$, $\ch_1E=0$). We will identify $H^0(X,\cZ)$ and $H^6(X,\cZ)$ with \cZ. We define the framing rank $v:\Ga\to\cZ$, $(v,\be,n)\to v$.
To deal with elements in $H^4(X,\cZ)$ we choose an ample class $\om\in H^2(X,\cZ)$. Then for any $\be\in H^4(X,\cZ)$ we have $\om\be\in H^6(X,\cZ)=\cZ$. We define a stability function $Z:\Ga\to\cC$ by the rule
$$Z(v,-\be,-n)=-n+i(\om\be+v).$$
Then an object $E\in\lA_X$ with $v(E)=1$ is ${+\infty}$-stable if and only if it is a PT morphism, \ie $E$ has the form $\lO_X\arr^f F$, where $F$ is pure of dimension $1$ and $\coker f$ is zero-dimensional.

Indeed, let $\cl(E)=(1,-\be,-n)$ and assume that $E$ is ${+\infty}$-stable. 
An object $E\in\lA_X$ with $v(E)=1$ has a filtration $E_0\sb E_1\sb E$ such that
$E_0,E/E_1\in\Coh_{\le1}[-1]$ and $E_1/E_0=\lO_X$ (see \cite[Lemma 3.11]{toda_curve}). Assume that $E/E_1=S[-1]$ has dimension one. Then $Z(E/E_1)=(0,-\be',-n')$ with $\be'\ne0$. Therefore
$$\mu_c(E)=\frac{c+n}{\om\be+1}>\frac{n'}{\om\be'}=\mu_c(E/E_1)$$
for $c\gg0$. This contradicts to ${+\infty}$-stability of $E$. Therefore $S$ has dimension zero. This implies that $$\Ext^1(E/E_1,E_1/E_0)=\Ext^1(S[-1],\lO_X)=\Ext^1(\lO_X,S)\dual=0$$ and we can rearrange our filtration in such a way that $E/E_1=0$. Then $E$ can be written as an object $f:\lO_X\to F$ in \lA with $E_0=F[-1]\sb E$.
If $F$ is not pure, then there exists a nonzero subsheaf $S\sb F$ of dimension $0$. Then, for the object $E'=S[-1]\sb E$ we have $\mu_c(E')=+\infty>\mu_c(E)$ and this contradicts to ${+\infty}$-stability of $E$. If $S=\coker f$ then $S[-1]$ is a quotient of $E$, and we have seen that this implies that $S$ has dimension zero.

Conversely, assume that $E=[\lO_X\arr^f F]$ is a PT morphism and let $E'\sb E$ be a proper subobject in $\lA_X$. Using an argument similar to the one described above we can show that $E'$ is an object in \lA. If $v(E')=0$ then $E'=S[-1]$ for the subsheaf $S\sb F$ of dimension one. Then $\mu_c(E')$ is finite and independent of $c$ and therefore $\mu_c(E')<\mu_c(E)$ for $c\gg0$. If $v(E')=1$ then $E/E'=S[-1]$ for a sheaf $S$ which is a quotient of $\coker f$. In particular, $S$ has dimension zero and therefore $\mu_c(E/E')=+\infty>\mu_c(E)$. This implies that $E$ is $+\infty$-stable.

Criterion of $+\infty$-stability proved in this example clearly differs from the criterion in Proposition \ref{infty stability}. And the reason for this is that there exist nonzero objects $E\in\lA_X$ such that $r(E)=0$.
\end{exm}

\section{Wall-crossing formulas}
\subsection{Relations in the Hall algebra}
\label{hall alg relations}
We use the same conventions as in the previous section. Results in this section will be formulated for the motivic Hall algebra of the ind-constructible abelian category. But they can be also formulated for the classical Ringel-Hall algebra of a finitary abelian category defined over a finite field.

\begin{dfn}
For any $\al\in\Ga_+$ with $v(\al)=0$ define
$$\tl \bb^Z_\al=[\lM_\al^\sst\to\lM]\in\hall_\al,$$
where $\lM_{\al}^\sst$ is the stack of $Z$-semistable objects having class \al. 
\rem{is independent of $c$ and we denote it by $\tl \bb_\al$.}
For any $c\in\cR\cup\set{\pm\infty}$ and $\al\in\Ga_+$ with $v(\al)=1$ define
$$\tl A^{Z}_{c,\al}=[\lM_{c,\al}^\sst\to\lM]\in\hall_\al,$$
where $\lM_{c,\al}^\sst$ is the stack of $Z_c$-semistable objects having class \al. 
\rem{$=\sum_{\over{M\ c\text{-sst}}{\cl(M)=\al}}[M]\in\hall.$}
Define similarly $\tl A_{c^\pm,\al}^Z\in\hall_\al$.
\end{dfn}

\begin{rmr}
If we would work over the classical Ringel-Hall algebra of a finitary category linear over a finite field, then we would define the element $\tl \bb_\al^Z$ as
$$\tl \bb_\al^Z=\sum_{\over{M\text{ is }Z\text{-sst}}{\cl(M)=\al}}[M].$$
The above definition is just a motivic analogue of this sum. 
\end{rmr}

\begin{dfn}
For any $\mu\in\cR$ define
$$\tl \bb^Z_\mu=\sum_{\over{\mu(\al)=\mu}{v(\al)=0}}\tl \bb_\al^Z\in\hall.$$
For any $c\in\cR$ and $\mu\in\cR$ define
$$\tl A_{c,\mu}^{Z}=\sum_{\over{\mu_c(\al)=\mu}{v(\al)=1}}\tl A_{c,\al}^Z\in\hall,\qquad
\tl A_{c^\pm,\mu}^{Z}=\sum_{\over{\mu_c(\al)=\mu}{v(\al)=1}}\tl A_{c^\pm,\al}^Z\in\hall.$$
Define
$$\tl A_{\pm\infty}^Z=\sum_{v(\al)=1} \tl A_{\pm\infty,\al}^Z\in\hall.$$
We will omit the index $Z$ if the stability function is clear from the context.
\end{dfn}

\begin{rmr}
The element $\tl A_{c^+,\mu}^{Z}$ should not be confused with $\tl A_{c+\eps,\mu}^{Z}$ for $0<\eps\ll1$. The first one is a sum over objects satisfying $\mu_c(E)=\mu$, while the second one is a sum over objects satisfying $\mu_{c+\eps}(E)=\mu$.
\end{rmr}

The following wall-crossing formula follows from Proposition \ref{main relation}.

\begin{thr}
\label{hall wall cross}
For any $c\in\cR$ and $\mu\in\cR$ we have
$$\tl A_{c,\mu}=\tl \bb_\mu \tl A_{c^-,\mu}=\tl A_{c^+,\mu}\tl \bb_\mu.$$
\end{thr}

\subsection{Relations in the quantum torus}
Let $\Ga^0=\ker (v:\Ga\to\cZ)$. We assume that $v:\Ga\to\cZ$ is surjective and there is a canonical splitting $\Ga=\Ga^0\oplus\cZ$. Moreover we assume that $\Ga_+=\Ga_+^0\xx\cN$, where $\Ga_+^0=\Ga_+\cap\Ga^0$.

\begin{rmr}
\label{A_*}
Consider the element $(0,1)\in\Ga_+^0\xx\cN=\Ga_+$.
It follows from our assumptions that there are no elements $0<\be<(0,1)$. This implies that $\tl A^Z_{c,(0,1)}$ is independent of $c\in\cR$ (and of stability function $Z$). We will denote it by $\tl A_*$.
\end{rmr}

\begin{rmr}
Let $\qt=\qt_{\Ga}$ be the motivic quantum torus and let \cqt be the completion of \qt with respect to $\Ga_+$.
Let $\cqt^0\sb \cqt$ be the subalgebra consisting of elements having their degrees in $\Ga^0_+$. We define $x_*=x^{(0,1)}\in \cqt$, where $(0,1)\in\Ga^0_+\xx\cN$.
\end{rmr}

Assume now that there is an algebra homomorphism
$$I:\hall\to \cqt$$
respecting the \Ga-grading. We will call it the integration map (later we will deal with several examples of such map).

\begin{dfn}
\label{invar}
For any $\al\in\Ga^0_+$ define unframed invariants
$$\bb_\al^Z =I(\tl \bb_\al^Z)/x^\al\in\lV.$$
Define $A_*=I(\tl A_*)/x_*\in\lV$.
For any $c\in\cR\cup\set{\pm\infty}$ and $\al\in\Ga^0_+$ define framed invariants $A_{c,\al}^Z\in\lV$ by
\rem{we don't define $A_{c^\pm,\al}$ as it is covered by $A_{c,\al}$}
$$A_{c,\al}^Z=I(\tl A_{c,(\al,1)}^Z)/I(\tl A_*)x^\al.$$
For any $c\in\cR$ and $\mu\in\cR$ define
$$B_\mu^Z=I(\tl B_\mu^Z)\in\cqt^0,\qquad A_{c,\mu}^Z=I(\tl A_{c,\mu}^Z)/I(\tl A_*)\in\cqt^0$$
and analogously define 
$A_{c^\pm,\mu}^Z,A_{\pm\infty}^Z\in\cqt^0$.
We will omit the index $Z$ if the stability function is clear from the context. 
\rem{we want actually $\frac{\cL-1}{-\cL^\oh}S_\nu A$. But application of $S_\nu$ will require changing of $B$, which is not good. So we just need to multiply with $\frac{\cL-1}{-\cL^\oh}$.}
\end{dfn}

\begin{dfn}
\label{def:nu}
For any group homomorphism $\la:\Ga\to \cZ$, define the algebra homomophism $$S_\la:\cqt\to \cqt,\qquad x^\al\mto (-q^\oh)^{\la(\al)}x^\al.$$
Define $\nu:\Ga\to\cZ$ by
$$\nu(\al)=\ang{\al,(0,1)},\qquad\al\in\Ga.$$
\end{dfn}

\begin{rmr}
It follows from
$$x^\al\circ x^\be=(-q^\oh)^{\ang{\al,\be}}x^{\al+\be}=q^{\ang{\al,\be}}x^\be\circ x^\al,\qquad \al,\be\in\Ga$$
that for any $f\in \cqt$ we have 
$$f\circ x^\be=x^\be S_{\ang{-,\be}}f=x^\be\circ S_{2\ang{-,\be}}f.$$
In particular
$$fx_*=S_{-\nu}f\circ x_*=x_*\circ S_{\nu}f.$$
\end{rmr}

\begin{thr}[Wall-crossing formula]
\label{wall-crossing thr}
For any $c\in\cR$ and $\mu\in\cR$ we have
$$A_{c,\mu}
=S_{\nu} \bb_\mu \circ A_{c^-,\mu}
=A_{c^+,\mu}\circ S_{-\nu} \bb_\mu$$
\end{thr}
\begin{proof}
It follows from Theorem \ref{hall wall cross} that
$$A_{c,\mu}x_*=\bb_\mu \circ A_{c^-,\mu}x_*=A_{c^+,\mu}x_*\circ \bb_\mu.$$
Expressions in this equations can be written as follows.
$$A_{c,\mu}x_*=S_{-\nu}A_{c,\mu}\circ x_*,$$
$$\bb_\mu \circ A_{c^-,\mu}x_*=\bb_\mu \circ S_{-\nu}A_{c^-,\mu}\circ x_*,$$
$$A_{c^+,\mu}x_*\circ \bb_\mu=S_{-\nu}A_{c^+,\mu}\circ x_*\circ \bb_\mu
=S_{-\nu}A_{c^+,\mu}\circ S_{-2\nu}\bb_\mu \circ x_*.$$
This implies
$$S_{-\nu}A_{c,\mu}=\bb_\mu \circ S_{-\nu}A_{c^-,\mu}
=S_{-\nu}A_{c^+,\mu}\circ S_{-2\nu}\bb_\mu .$$
Now we apply the operator $S_{\nu}$.
\end{proof}

\subsection{Uniform notation}
In this section we will formulate our results using a version of notation by Nagao and Nakajima \cite{nagao_counting}. Our wall-crossing formulas have a particularly nice form in this notation. We will always assume that stability functions considered in this section have the HN property.

Let $Z=-d+ir$ be a stability function.
For any $\al\in\Ga_+^0$, $c\in\cR$, we have defined invariants $A_{c,\al}^Z,\bb^Z_\al\in\lV$.

\begin{dfn}
For any $\al\in\Ga_+^0$ define the framed invariant
$$A^Z_\al:=A^Z_{-d(\al),\al}\in\lV.$$
Equivalently, $A^Z_\al=A^Z_{c,\al}$, where $c\in\cR$ is such that $\mu_{Z_c}(\al,1)=0$. Define the series
$$A^Z=\sum_{\al\in\Ga_+^0} A^Z_{\al}x^\al\in\cqt^0.$$
\end{dfn}

For any $a\in\cR$ define the new stability function $$Z^a=Z+ar.$$
The corresponding slope function is given by
$$\mu_{Z^a}(\al)=\frac{d(\al)-ar(\al)}{r(\al)}=\mu_Z(\al)-a.$$

\begin{rmr}
Of course $Z^a$-stability is equivalent to $Z$-stability. But the series $A^a:=A^{Z^a}$ depends on $a$. For any $a\in\cR$, $\al\in\Ga_+^0$ we have 
$$A^{Z^a}_\al=A^{Z^a}_{c,\al}=A^Z_{c,\al},$$
where $c=ar(\al)-d(\al)$. 
Assume that $r(\al)>0$ for any $0\ne\al\in\Ga_+^0$. If $a\gg0$ then $c\gg0$ and we have then $A^{Z^a}_\al=A^Z_{+\infty,\al}$. In the same way for $a\ll0$ we have $A^{Z^a}_\al=A_{-\infty,\al}^Z$. Note that $A_{\pm\infty}^Z$ are independent of $Z$ by Proposition \ref{infty stability}.
\end{rmr}

Our goal is to study the behavior of the series $A^a=A^{Z^a}$ when we slightly shift $a$. In the same way as before, we can show that for any $\al\in\Ga_+^0$ the invariant $A^{{a+\eps}}_{\al}$ is independent of \eps if $0<\eps\ll1$. We denote it by $A^{{a^+}}_{\al}$ and define $A^{{a^+}}=\sum A^{{a^+}}_{\al}x^\al\in\cqt^0$. In the same way we define $A^{{a^-}}_{\al}\in\lV$ and $A^{{a^-}}\in\cqt^0$.

\begin{thr}
\label{NN2}
Assume that for any $0\ne\al\in\Ga_+^0$ we have $r(\al)>0$. Then for any $a\in\cR$ we have
$$A^{{a}}=S_\nu \bb^Z_a\circ A^{{a^-}}=A^{{a^+}}\circ S_{-\nu}\bb^Z_a,$$
where $B_a^Z=\sum_{\mu_Z(\al)=a}B_\al^Zx^\al\in\cqt^0$.
\end{thr}
\begin{proof}
Applying the wall-crossing formula (see Theorem \ref{wall-crossing thr}) to the stability function $Z^a$, $\mu=0$, and arbitrary $c\in\cR$,
we obtain
$$A_{c,\mu}^{Z^a}
=S_{\nu} \bb_\mu^{Z^a} \circ A_{c^-,\mu}^{Z^a}
=A_{c^+,\mu}^{Z^a}\circ S_{-\nu} \bb_\mu^{Z^a}.$$
We know that $Z^a$-stability is equivalent to $Z$-stability. If $\al\in\Ga_+^0$ is such that $\mu_{Z^a}(\al)=\mu=0$ then 
$$ar(\al)-d(\al)=0$$
and therefore $\mu_Z(\al)=a$. This implies that we can identify $\bb_\mu^{Z^a}$ with $\bb^Z_a$.

We can identify $A_{c,\mu}^{Z^a}$ with a summand of $A^{Z^{a}}$ corresponding to those $\al\in\Ga_+^0$ that satisfy
\begin{equation}
ar(\al)-d(\al)=c.
\label{NN1}
\end{equation}

The series $A_{c^-,\mu}^{Z^a}$ and $A_{c^+,\mu}^{Z^a}$ also have coefficients satisfying \eqref{NN1}. We just have to show that stability with respect to $Z^a_{c^+}$ (resp.\ $Z^a_{c^-}$) is equivalent to the stability imposed in the invariants $A^{Z^{a+\eps}}$ (resp.\ $A^{Z^{a-\eps}}$) for $0<\eps\ll1$.
In order to define $A^{Z^{a+\eps}}_{\al}$ for a given $\al$ satisfying \eqref{NN1}, we should choose $b\in\cR$ such that $\mu_{Z^{a+\eps}_{b}}(\al,1)=0$, or equivalently,
$$b=(a+\eps)r(\al)-d(\al)=c+\eps r(\al)$$
and then take $A^{Z^{a+\eps}}_{b,\al}$. The corresponding stability function is
$Z^{a+\eps}_b$ and $Z^{a+\eps}_b$-stability is equivalent to $Z^{a}_b$-stability. We note now that $0<\eps r(\al)\ll1$ if $0<\eps\ll1$ and therefore $Z^a_b$-stability is equivalent to $Z^a_{c^+}$-stability.
\end{proof}

\begin{rmr}
The relevance of the above wall-crossing formula is that, in contrast to Theorem \ref{wall-crossing thr}, we don't have to get track of the changing slopes if we want to apply our wall-crossing formula several times. This makes the wall-crossing process much easier.
\end{rmr}

\begin{rmr}
Invariants $A^{Z^a}_{\al}$ contain information about all the invariants $A^{Z}_{c,\al}$. Indeed, given $c\in\cR$ and $\al\in\Ga_+^0$, let $a=\mu_{Z_c}(\al)$.
Then $c$ has the property $\mu_{Z^a_c}(\al,1)=0$ and therefore $A^{Z^a}_{\al}=A^{Z^a}_{c,\al}$. On the other hand $Z_c$-stability is equivalent to $Z^a_c$-stability and therefore $A^{Z^a}_{\al}=A^{Z^a}_{c,\al}=A^{Z}_{c,\al}$.
\end{rmr}

\begin{crl}
\label{NN3}
Assume that for any $0\ne\al\in\Ga_+^0$ we have $r(\al)>0$. Then for any $a\in\cR$ we have
\begin{multline*}
A^a=S_\nu\left( \prod^{\leftarrow}_{b\le a}\bb^Z_b\right)\circ A_{-\infty}\circ S_{-\nu}\left(\prod^{\leftarrow}_{b< a}\bb^Z_b\right)\inv\\
=S_\nu\left( \prod^{\leftarrow}_{b> a}\bb^Z_b\right)\inv\circ A_{+\infty}\circ S_{-\nu}\left(\prod^{\leftarrow}_{b\ge a}\bb^Z_b\right),	
\end{multline*}
where the products are taken in the decreasing order of $b$.
\end{crl}
\begin{proof}
Let us prove just the first equation. It follows from Theorem \ref{NN2} that
$$A^{{a^+}}=S_\nu \bb^Z_a\circ A^{{a^-}}\circ S_{-\nu}(\bb^Z_a)\inv.$$
Crossing all the walls in the interval $(-\infty,a]$ we obtain
$$A^{{a^+}}=S_\nu\left( \prod^{\leftarrow}_{b\le a}\bb^Z_b\right)\circ A_{-\infty}\circ S_{-\nu}\left(\prod^{\leftarrow}_{b\le a}\bb^Z_b\right)\inv.$$
Now we just apply $A^a=A^{a^+}S_{-\nu}(\bb^Z_a)$.
\end{proof}

\subsection{Symmetric case}
In this section we will assume that the skew-symmetric form $\ang{-,-}$, restricted to $\Ga^0$, is zero. This is equivalent to the assumption that $\cqt^0$ is commutative.
Note that in this case we can reconstruct the skew-symmetric form $\ang{-,-}$ from $\nu|_{\Ga^0}$ by the formula
$$\ang{(\al,k),(\be,l)}=l\nu(\al)-k\nu(\be).$$
Let us slightly rewrite Theorem \ref{wall-crossing thr}.
\begin{thr}
\label{wall-cross symmetric}
For any $c\in\cR$, $\mu\in\cR$ the following equation holds in the commutative algebra $\cqt^0$
$$A_{c^+,\mu}=C_\mu A_{c^-,\mu},$$
where the transfer series $C_\mu\in\cqt^0$ is defined by
$$C_\mu=S_{\nu}\bb_\mu\cdot S_{-\nu}\bb_\mu\inv=S_{-\nu}(S_{2\nu}\bb_\mu/\bb_\mu).$$
\end{thr}
\rem{old notation: $B'_\mu=C_\mu\inv$}

\begin{rmr}
According to this result we just need to know the transfer series $C_\mu$ for all $\mu\in\cR$ and either the series $A_{-\infty}$ or $A_{+\infty}$ in order to determine the series $A_{c,\mu}$ for any $c\in\cR$ and $\mu\in\cR$. A more precise statement will be given in Theorem~\ref{thr:asympt}.
\end{rmr}

All the results of the previous sections could be proved also for classical Hall algebras and for the quantum torus over $\cQ(q^\oh)$.
In the following definition we need the motivic quantum torus, as we will use its \la-ring structure.

\begin{dfn}
\label{DT inv}
For any $\mu\in\cR$ we define the Donaldson-Thomas invariants 
$\Om_\mu=\sum_{\mu(\al)=\mu}\Om_\al x^\al\in \cqt^0$ by the formula
$$\bb_\mu=\Exp\left(\frac{\Om_\mu}{\cL-1}\right),$$
where $\Exp$ is the plethystic exponent on the complete \la-ring $\cqt^0$ \cite{getzler_mixed,mozgovoy_computational}.
\end{dfn}
Note that
$$C_\mu=S_{\nu}\bb_\mu\cdot S_{-\nu}\bb_\mu\inv
=S_{-\nu}\Exp\left(\frac{S_{2\nu}\Om_\mu-\Om_\mu}{\cL-1}\right).$$
\rem{$B'_{>\mu}=\prod_{\eta>\mu}B'_\eta
=S_{-\nu}\Exp\left(\frac{\Om_{>\mu}-S_{2\nu}\Om_{>\mu}}{\cL-1}\right).$}

\begin{rmr}
It is important to note that the algebra homomorphism $S_\nu:\cqt\to\cqt$ is not a \la-ring homomorphism in general. Indeed, we have
$$S_\nu(\psi_n(x^\al))=S_\nu(x^{n\al})=(-1)^{n\nu(\al)}\cL^{n\nu(\al)/2}x^{n\al}$$
while
$$\psi_n(S_\nu x^\al)
=\psi_n((-1)^{\nu(\al)}\cL^{\nu(\al)/2}x^\al)=(-1)^{\nu(\al)}\cL^{n\nu(\al)/2}x^{n\al}.
$$
On the other hand $S_{2\nu}:\cqt\to\cqt$ is a \la-ring homomorphism. Therefore, we can interchange the operators $S_{2\nu}$ and $\Exp$.
\end{rmr}

\begin{rmr}
Assume that for any \al there exists the Euler number specialization of $\Om_\al\in\lV$ (we put $\cL^\oh=1$), which we denote by $\ub\Om_\al$ (these numbers are called classical Donaldson-Thomas invariants). Then the above formula implies that there exists the Euler number specialization of $C_\mu$
$$\ub C_\mu=\ub S_{\nu}\Exp\left(\sum_{\mu(\al)=\mu}\nu(\al)\ub\Om_\al x^\al\right),$$
where the operator $\ub S_\nu$ is defined by $\ub S_{\nu}(x^\al)=(-1)^{\nu(\al)}x^\al$.
\end{rmr}


\begin{dfn}
Let $c\in\cR$, $\mu\in\cR$. For any element $f=\sum_\al f_\al x^\al\in\cqt^0$, we define its truncation $\tau^c_\mu f\in\cqt^0$ by the formula
$$\tau^c_\mu f=\sum_{\mu_c(\al,1)=\mu}f_\al x^\al.$$
\end{dfn}


\begin{thr}
\label{thr:asympt}
For any $c\in\cR,\mu\in\cR$ we have
\begin{enumerate}
	\item $A_{c^+,\mu}=\tau^c_\mu((C_{>\mu})\inv A_{+\infty})$,
	\item $A_{c^-,\mu}=\tau^c_\mu(C_{<\mu} A_{-\infty})$,
\end{enumerate}
where
$$C_{>\mu}=\prod_{\eta>\mu}C_\eta,\qquad C_{<\mu}=\prod_{\eta<\mu}C_\eta.$$
\end{thr}
\begin{proof}
We will prove only the first statement.
To simplify our notation let $D_\mu=C_\mu\inv$. We will write $\mu_c(\al)$ for $\mu_c(\al,1)$ if $\al\in\Ga^0$. The non-twisted slope function (corresponding to the original stability function $Z$) still will be denoted by $\mu(\al)$ for $\al\in\Ga^0$. Our statement says that for any $0<\al\in\Ga^0$ we have
\begin{equation}
A_{c^+,\al}=\sum_{\over{\mu_c(\al)<\mu(\al_1)<\dots<\mu(\al_n)}{\al_i>0,\al'=\al-\sum\al_i\ge0}}D_{\al_1}\dots D_{\al_n}A_{+\infty,\al'}.
\label{eq:1}
\end{equation}

For any sequence $(\al_1,\dots,\al_n)$ as in the above sum define
\begin{equation}
\al_k'=\al-\sum_{i=1}^k\al_i,\qquad 1\le k\le n
\label{eq:2}
\end{equation}
and let $c_k\in\cR$ be the numbers uniquely determined by the condition \begin{equation}
\mu(\al_k)=\mu_{c_k}(\al_k').
\label{eq:3}
\end{equation}
Note that $\al_{k-1}'=\al_{k}+\al_{k}'$ and therefore
$$\mu_{c_k}(\al_k')=\mu(\al_k)=\mu_{c_k}(\al_{k-1}').$$
This implies
$$\mu_{c_{k-1}}(\al_{k-1}')=\mu(\al_{k-1})<\mu(\al_{k})=\mu_{c_{k}}(\al_{k-1}')$$
and therefore $c_{k-1}<c_{k}$. It follows from
$$\mu_c(\al)<\mu(\al_1)=\mu_{c_1}(\al_1')=\mu_{c_1}(\al)$$
that $c<c_1$. Therefore
$$c<c_1<\dots<c_n.$$
Conversely, let there be given such a sequence of numbers and a sequence of elements $\al_1,\dots,\al_n>0$ in $\Ga^0$ such that $\al-\sum\al_i\ge0$ and such that equation \eqref{eq:3} is satisfied, where $\al_k'$ are defined by equation \eqref{eq:2}. Then one can show that
$$\mu_c(\al)<\mu(\al_1)<\dots<\mu(\al_n).$$

The wall crossing formula says that for $a\in\cR$, $\mu\in\cR$
$$A_{a^-,\mu}=D_\mu\cdot A_{a^+,\mu}.$$
Equivalently,
$$A_{a^-,\al}=\sum_{\over{0\le\be\le\al}{\mu(\be)=\mu_a(\al)}}D_\be A_{a^+,\al-\be},$$
where we always allow the summand corresponding to $\be=0$.
Equation \eqref{eq:1} is obtained by applying this wall-crossing formula at points
$$c_1<\dots <c_n$$
for all possible choices $c<c_1<\dots<c_n$.
\end{proof}

The following result is a generalization of \cite[Section 2.4]{chuang_motivic}

\begin{crl}
For any $c\in\cR$ and $\mu\in\cR$ we have
$$\tau^c_\mu(C_{\le\mu}A_{-\infty})
=\tau^c_\mu((C_{>\mu})\inv A_{+\infty}) 
.$$
\end{crl}
\begin{proof}
Just note that
$$A_{c^+,\mu}=C_\mu A_{c^-,\mu}
=C_\mu\tau^c_\mu(C_{<\mu} A_{-\infty})
=\tau^c_\mu(C_{\le\mu} A_{-\infty}).
$$
\end{proof}
\rem{This implies the recursion formula
$(A_{+\infty} B'_{<\mu})_\mu
=(A_{+\infty} B'_{\le-\mu})_{-\mu}(s\inv).$}

\section{Applications}
\subsection{Smooth models}
Let $Q=(Q_0,Q_1)$ be a quiver. Let $w\in\cN^{Q_0}$ and let $Q'$ be a new quiver obtained from $Q$ by adding one vertex $*$ and $w_i$ arrows from $*$ to $i$ for every $i\in Q_0$.  Define $\Ga^0=\cZ^{Q_0}$ and $\Ga=\cZ^{Q'_0}=\Ga^0\oplus\cZ$.
Define the framing rank $v:\Ga\to\cZ$ by 
$$v(\al,\al_*)=\al_*,\qquad (\al,\al_*)\in\Ga^0\oplus\cZ.$$
Let $\hi$ be the Euler-Ringel form of the quiver $Q'$ defined by
$$\hi(\al,\be)=\sum_{i\in Q'_0}\al_i\be_i-\sum_{a:i\to j}\al_i\be_j,\qquad \al,\be\in\Ga.$$
We define the tits form of $Q'$ by $T(\al)=\hi(\al,\al)$ and the skew symmetric form $\ang{-,-}$ of $Q'$ by
$$\ang{\al,\be}=\hi(\al,\be)-\hi(\be,\al),\qquad \al,\be\in\Ga.$$
Note that the homomorphism $\nu$ from Definition \ref{def:nu} is given by
\begin{equation}
\nu(\al)=\ang{(\al,0),(0,1)}=\sum_{i\in Q_0}w_i\al_i=w\cdot\al,\qquad \al\in\Ga^0.
\label{eq:nu}
\end{equation}
Let $\lA$ be the category of $Q'$-representations (over \cC) and let the homomorphism $\cl:K_0(\lA)\to\Ga$ be given by $\cl(M)=\lb\dim M$, where for any representation $M$ of $Q'$ we define it dimension vector $\lb\dim M$ by 
$$\lb\dim M=(\dim M_i)_{i\in Q'_0}\in\Ga.$$
Let 
$$H=H(\lA)=\oplus_{\al\in\Ga_+}H(\lA)_\al$$
be the motivic Hall algebra of \lA and let \hall be its completion.
Let $\cqt$ be the complete motivic quantum torus of \Ga. The integration map $I:\hall\to\cqt$ was described in Proposition~\ref{reineke map}. 

Given $\te\in\cR^{Q_0}$, we define the corresponding stability function $Z=-d+ir:\Ga\to\cC$ by
$$d(\al,\al_*)=\te\cdot\al,\qquad r(\al,\al_*)=\sum_{i\in Q_0}\al_i+\al_*,\qquad (\al,\al_*)\in\Ga=\Ga^0\oplus\cZ.$$
For any $c\in\cR$ let $\mu_c$ be the slope function corresponding to the stability function $Z_c=Z-cv$, \ie
$$\mu_c(\al,\al_*)=\frac{\te\cdot\al+c\al_*}{\sum_{i\in Q_0}\al_i+\al_*}.$$

\begin{rmr}
We know from Proposition \ref{infty stability} that the notions of $\pm\infty$-stability are independent of $\te$. If a representation $M'=(M,\cC)$ of $Q'$ is $-\infty$-stable then for any its proper subrepresentation $N$ we should have $v(N)=1$. But $M'$ has an obvious subrepresentation $(M,0)$, which implies that $M=0$. This means that there is the unique $-\infty$-stable representation $(0,\cC)$.
\end{rmr}

Let $\mu\in\cR$ and let $\al\in\cN^{Q_0}$ be such that $\mu(\al)=\mu$. Let $\lM^\te_{\al}=\lM^\te(Q,\al)$ be the moduli stack of $Z$-semistable representations of $Q$ of dimension \al. 
The series $\bb_{\mu}=\bb^\te_\mu\in\cqt^0$ defined in \ref{invar} is given by
$$\bb_{\mu}=\sum_{\mu(\al)=\mu}(-\cL^\oh)^{T(\al)}[\lM^{\te}_\al] x^\al.$$
There is just one representation of dimension $(0,1)$. Therefore the invariant $A_*\in\lV$ (see Definition \ref{invar}) is given by
$$A_*=\frac{-\cL^\oh}{\cL-1}.$$
Let $\al'=(\al,1)\in\Ga$, $c=\mu$, and let $\lM^\te_{\al,w}=\lM^{(\te,c^+)}(Q',\al')$ be the moduli stack of $Z_{c^+}$-stable representations of the quiver $Q'$ of dimension $\al'$. It is called the smooth model of the moduli stack $\lM^\te_\al$ in \cite{engel_smooth}.
The series $A_{c^+,\mu}=A_{c^+,\mu}^\te\in\cqt$ defined in \ref{invar} is given by (note that $T(\al')=T(\al)+1-w\cdot\al$)
\begin{multline}
\label{eq:4}
A_{c^+,\mu}
=A_*\inv\sum_{\mu(\al)=\mu}(-\cL^\oh)^{T(\al)+1-w\cdot\al}[\lM^\te_{\al,w}]x^\al\\
=(\cL-1)S_{-\nu}\sum_{\mu(\al)=\mu} (-\cL^\oh)^{T(\al)}[\lM^\te_{\al,w}]x^\al.
\end{multline}
Consider the moduli space $\lM^{(\te,c^-)}(Q',\al')$ of $Z_{c^-}$-stable representations of the quiver $Q'$ of dimension $\al'$ (as before we consider $c=\mu$ and $\al'=(\al,1)$, where $\al\in\Ga^0$ is such that $\mu(\al)=\mu$). One can easily see that the notion of $c^-$-stability of representations of dimension \al is equivalent to the notion of $-\infty$-stability. Therefore the above moduli stack is empty unless $\al=0$ in which case it consists of one point. This implies
$I(\tl A_{c^-,\mu})=A_* x_*$ and therefore
\begin{equation}
A_{c^-,\mu}=1.
\label{eq:5}
\end{equation}
It follows from Theorem \ref{wall-crossing thr} that
$$A_{c^+,\mu}=S_{\nu} \bb_\mu \circ A_{c^-,\mu} \circ S_{-\nu} \bb_\mu\inv.$$
Combining this formula with eqautions \eqref{eq:4}, \eqref{eq:5} we obtain the following result equivalent to \cite[Theorem 5.2]{engel_smooth}

\begin{thr}
For any $\te\in\cR^{Q_0}$, $\mu\in\cR$, and $w\in\cN^{Q_0}$ we have
$$\sum_{\mu(\al)=\mu} (-\cL^\oh)^{T(\al)}[\lM^\te_{\al,w}]x^\al
=\frac{1}{\cL-1} S_{2\nu} \bb_\mu^\te \circ (\bb_\mu^\te)\inv.$$
\end{thr}

\subsection{Quivers with potentials}
Let $(Q,W)$ be a quiver with polynomial potential and let $\wt:Q_1\to\cN$ be a map such that $W$ is homogeneous of degree $1$ with respect to $\wt$. Let $J_{Q,W}=\cC Q/(\dd W)$ be the corresponding Jacobian algebra.

For a given $w\in\cN^{Q_0}$, we define a new quiver $Q'$ in the same way as before.
Potential $W$ can be considered also as a potential on $Q'$. Let $J_{Q',W}=\cC Q'/(\dd W)$ be the corresponding Jacobian algebra.
We define the group \Ga, bilinear forms \hi and $\ang{-,-}$ on \Ga, stability function $Z:\Ga\to\cZ$ (for a fixed $\te\in\cR^{Q_0}$), the framing rank function $v:\Ga\to\cZ$, the Hall algebra $\hall$ of $Q'$, and the quantum torus $\cqt$ of $Q'$ in the same way as before.

Using the weight function $\wt:Q_1\to\cN$ one can construct \cite{mozgovoy_motivica} an integration map $\Ieq:\heq\to \cqt$ (which is an algebra homomorphism) from the equivariant Hall algebra $\heq\sb\hall$ to the quantum torus. The Donaldson-Thomas series $\bb_\mu^\te,A_{c,\mu}^\te\in\cqt^0$ of the moduli spaces of representations of the Jacobian algebra are defined by
$$\bb_\mu^\te=\Ieq(\tl \bb_\mu),\qquad A_{c,\mu}^\te =\Ieq(\tl A_{c,\mu})/\Ieq(\tl A_*),$$
where $\tl \bb_\mu$, $\tl A_{c,\mu}$ are the elements in the Hall algebra of the quiver $Q'$ defined in Section \ref{hall alg relations}. In the same way we define the elements $A_{\pm\infty}\in\cqt^0$ (recall that by Proposition \ref{infty stability} they are independent of \te). 

\begin{dfn}
Define the universal Donaldson-Thomas series $\bb_{U}\in\cqt^0$ to be $\bb^\te_\mu$ for $\te=0$, $\mu=0$.
\end{dfn}

\begin{rmr}
Representations counted by $\bb_{U}$ are all possible representations of $J_{Q,W}$ because the imposed stability condition is trivial and all objects are semistable.
\end{rmr}

\begin{dfn}
Define the non-commutative Donaldson-Thomas invariants
$$A^w_{NCDT}=S_\nu A_{+\infty}\in\cqt^0.$$
\end{dfn}

\begin{rmr}
Representations counted by the invariant $A^w_{NCDT}$ are $+\infty$-stable representations of $J_{Q',W}$. By Proposition \ref{infty stability} these are representations $M'=(M,\cC)$ such that $\cC$ generates the whole representation $M'$, \ie if a subrepresentation $N'\sb M'$ satisfies $v(N')=1$ then $N'=M'$. Such representations were called cyclic in \cite{mozgovoy_noncommutative,szendroi_non-commutative}.
\end{rmr}

We can express invariants $A^w_{NCDT}$ in terms of $\bb_{U}$ (\cf \cite[Theorem 7.1]{morrison_motivic}).

\begin{prp}
We have
$$A^w_{NCDT}=S_{2\nu}\bb_{U}\circ \bb\inv_{U}.$$
\end{prp}
\begin{proof}
Let $\te=0$, $\mu=0$, and $c=0$. Then $A^\te_{c^-,\mu}=A_{-\infty}$, $A^\te_{c^+,\mu}=A_{+\infty}$ and $B^\te_\mu=\bb_U$. 
The only $-\infty$-stable representation has dimension $(0,1)$ and therefore $\tl A_{-\infty}=\tl A_*$ and
$$A_{c^-,\mu}^\te=A_{-\infty}=\Ieq(\tl A_{-\infty})/\Ieq(\tl A_*)=1.$$
Using the wall-crossing formula (Theorem \ref{wall-crossing thr}) we obtain
$$A_{+\infty}=A_{c^+,\mu}^\te=S_{\nu} \bb_\mu^\te \circ A_{c^-,\mu}^\te\circ (S_{-\nu} \bb_\mu^\te)\inv
=S_{\nu} \bb_\mu^\te \circ (S_{-\nu} \bb_\mu^\te)\inv$$
and the statement of the proposition follows.
\end{proof}

\begin{rmr}
We have seen that for any $\al\in\Ga^0$ we have $\nu(\al)=w\cdot\al$.
The above proposition implies that for any $w\in\cN^{Q_0}$ we can express $A^w_{NCDT}$ in terms of $\bb_{U}$. Conversely, if we know $A^w_{NCDT}$ for $w$ from the standard basis of $\cZ^{Q_0}$, then we can reconstruct $\bb_{U}$ and therefore also $A^w_{NCDT}$ for arbitrary $w$.
\end{rmr}

\begin{rmr}
\label{reconstr2}
It follows from \cite[Thorem 5.5]{mozgovoy_motivica} (which is a version of the result of Reineke \cite{reineke_harder-narasimhan}) that, for arbitrary $\te\in\cR^{Q_0}$ and $\mu\in\cR$, we can reconstruct $\bb^{\te}_\mu$ from $\bb_{U}$ in a quite explicit way.
On the other hand, by the previous proposition we can reconstruct $A_{+\infty}$ from $\bb_{U}$ (and $A_{-\infty}=1$, as we have seen). 
But if we know $\bb^\te_\mu$ for all $\mu\in\cR$ then using the wall-crossing formula (Theorem \ref{wall-crossing thr}) we can move from $A_{-\infty}$ (or $A_{+\infty}$) to $A_{c,\mu}^\te$ (see Theorem \ref{thr:asympt} in the symmetric case or Corollary \ref{NN3} in general). This means that we can also reconstruct the series $A_{c,\mu}^\te$ from $\bb_{U}$ for an arbitrary stability parameter $\te\in\cR^{Q_0}$, arbitrary parameter $c\in\cR$ and arbitrary slope $\mu\in\cR$. This is why we call the series $\bb_{U}$ universal.
\end{rmr}

Let us give now a couple of examples, where we can explicitly compute $\bb_{U}$ (and therefore also $A^\te_{c,\mu}$ according to Remark \ref{reconstr2}).
The following result was proved in \cite{behrend_motivic}.

\begin{thr}[Example: {$\cC^3$}]
Let $Q$ be a quiver with one vertex and three loops $x,y,z$. Let $W=xyz-xzy$. Then
$$\bb_{U}=\Exp\left(\frac{\cL^2}{\cL-1}\sum_{n\ge1}x^n\right).$$
\end{thr}

The following result will be proved in \cite{morrison_motivica}

\begin{thr}[Example: Conifold]
Let $Q$ be a quiver with vertices $1,2$ and arrows $a_1,a_2:1\to2$, $b_1,b_2:2\to1$. Let $W=a_1b_1a_2b_2-a_1b_2a_2b_1$. Then
$$\bb_{U}
=\Exp\left(\frac{(\cL+\cL^2)x_1x_2-\cL^\oh(x_1+x_2)}{\cL-1}\sum_{n\ge0}(x_1x_2)^n\right).
$$
\end{thr}

\begin{rmr}
Using Definition \ref{DT inv} we can express motivic Donaldson-Thomas invariants of the Jacobian algebra $J_{Q,W}$ as
$$\sum_{\al\in\cN^2}\Om_\al x^\al=\left((\cL+\cL^2)x_1x_2-\cL^\oh(x_1+x_2)\right)\sum_{n\ge0}(x_1x_2)^n.$$
In particular $\Om_\al\in\cZ[\cL^\oh]$ and $\Om_\al(-\cL^\oh)\in\cN[\cL^\oh]$.
The author expects that the property $\Om_\al(-\cL^\oh)\in\cN[\cL^{\pm\oh}]$ should hold for a broad class of symmetric quivers with potentials. This property for symmetric quivers with the trivial potential was conjectured by Kontsevich and Soibelman \cite{kontsevich_cohomological} and was proved  by Efimov \cite{efimov_cohomological}.
\end{rmr}

\bibliography{../tex/papers}

\providecommand{\bysame}{\leavevmode\hbox to3em{\hrulefill}\thinspace}
\providecommand{\href}[2]{#2}
\begin{thebibliography}{10}

\bibitem{aganagic_wall}
Mina Aganagic, Hirosi Ooguri, Cumrun Vafa, and Masahito Yamazaki, \emph{{W}all
  {C}rossing and {M}-theory}, 2009, \eprint{arxiv}{0908.1194}.

\bibitem{behrend_motivic}
Kai Behrend, Jim Bryan, and Balazs Szendroi, \emph{{M}otivic degree zero
  {Donaldson-Thomas} invariants}, 2009, \eprint{arxiv}{0909.5088}.

\bibitem{behrend_motivica}
Kai Behrend and Ajneet Dhillon, \emph{{O}n the motivic class of the stack of
  bundles}, Adv. Math. \textbf{212} (2007), no.~2, 617--644.

\bibitem{bridgeland_stability}
Tom Bridgeland, \emph{{S}tability conditions on triangulated categories}, Ann.
  of Math. (2) \textbf{166} (2007), no.~2, 317--345,
  \eprint{arxiv}{math/0212237}.

\bibitem{bridgeland_introduction}
Tom Bridgeland, \emph{{A}n introduction to motivic {H}all algebras}, 2010,
  \eprint{arxiv}{1002.4372}.

\bibitem{chuang_chamber}
{Wu-Yen} Chuang, {Duiliu-Emanuel} Diaconescu, and Guang Pan, \emph{{C}hamber
  {S}tructure and {W}allcrossing in the {ADHM} {T}heory of {C}urves {II}},
  2009, \eprint{arxiv}{0908.1119}.

\bibitem{chuang_motivic}
\bysame, \emph{{M}otivic {W}allcrossing and {C}ohomology of {T}he {M}oduli
  {S}pace of {H}itchin {P}airs}, 2010, \eprint{arxiv}{1004.4195}.

\bibitem{banorollin_motive}
Sebastian del Ba{\~n}o~Rollin and Vicente Navarro~Aznar, \emph{{O}n the motive
  of a quotient variety}, Collect. Math. \textbf{49} (1998), no.~2-3, 203--226,
  Dedicated to the memory of Fernando Serrano.

\bibitem{diaconescu_chamber}
{Duiliu-Emanuel} Diaconescu, \emph{{C}hamber {S}tructure and {W}allcrossing in
  {T}he {ADHM} {T}heory of {C}urves {I}}, 2009, \eprint{arxiv}{0904.4451}.

\bibitem{efimov_cohomological}
Alexander~I. Efimov, \emph{{C}ohomological {H}all algebra of a symmetric
  quiver}, 2011, \eprint{arxiv}{1103.2736}.

\bibitem{engel_smooth}
Johannes Engel and Markus Reineke, \emph{{S}mooth models of quiver moduli},
  Math. Z. \textbf{262} (2009), no.~4, 817--848, \eprint{arxiv}{0706.4306}.

\bibitem{getzler_mixed}
Ezra Getzler, \emph{{M}ixed {H}odge structures of configuration spaces},
  Preprint 96-61, Max Planck Institute for Mathematics, Bonn, 1996,
  \eprint{arxiv}{alg-geom/9510018}.

\bibitem{gholampour_counting}
Amin Gholampour and Yunfeng Jiang, \emph{{C}ounting invariants for the {ADE}
  {M}c{K}ay quivers}, 2009, \eprint{arxiv}{0910.5551}.

\bibitem{gillet_descent}
H.~Gillet and C.~Soul{\'e}, \emph{{D}escent, motives and {$K$}-theory}, J.
  Reine Angew. Math. \textbf{478} (1996), 127--176,
  \eprint{arxiv}{alg-geom/9507013}.

\bibitem{guill'en_un}
Francisco Guill{\'e}n and Vicente Navarro~Aznar, \emph{{U}n crit\`ere
  d'extension des foncteurs d\'efinis sur les sch\'emas lisses}, Publ. Math.
  Inst. Hautes \'Etudes Sci. \textbf{95} (2002), 1--91.

\bibitem{heinloth_note}
Franziska Heinloth, \emph{{A} note on functional equations for zeta functions
  with values in {C}how motives}, Ann. Inst. Fourier (Grenoble) \textbf{57}
  (2007), no.~6, 1927--1945, \eprint{arxiv}{math/0512237}.

\bibitem{huybrechts_framed}
D.~Huybrechts and M.~Lehn, \emph{{F}ramed modules and their moduli}, Internat.
  J. Math. \textbf{6} (1995), no.~2, 297--324.

\bibitem{joyce_configurationsa}
Dominic Joyce, \emph{{C}onfigurations in abelian categories. {II}.
  {R}ingel-{H}all algebras}, Adv. Math. \textbf{210} (2007), no.~2, 635--706,
  \eprint{arxiv}{math.AG/0503029}.

\bibitem{joyce_theory}
Dominic Joyce and Yinan Song, \emph{{A} theory of generalized
  {D}onaldson-{T}homas invariants}, 2008, \eprint{arxiv}{0810.5645}.

\bibitem{kontsevich_stability}
Maxim Kontsevich and Yan Soibelman, \emph{{S}tability structures, motivic
  {D}onaldson-{T}homas invariants and cluster transformations}, 2008,
  \eprint{arxiv}{0811.2435}.

\bibitem{kontsevich_cohomological}
\bysame, \emph{{C}ohomological {H}all algebra, exponential {H}odge structures
  and motivic {D}onaldson-{T}homas invariants}, 2010,
  \eprint{arxiv}{1006.2706}.

\bibitem{morrison_motivic}
Andrew Morrison, \emph{{M}otivic invariants of quivers via dimensional
  reduction}, 2011, \eprint{arxiv}{1103.3819}.

\bibitem{morrison_motivica}
Andrew Morrison, Sergey Mozgovoy, Kentaro Nagao, and Balazs Szendroi,
  \emph{{M}otivic {D}onaldson-{T}homas invariants of the conifold and the
  refined topological vertex}, to appear.

\bibitem{mozgovoy_computational}
Sergey Mozgovoy, \emph{{A} computational criterion for the {K}ac conjecture},
  J. Algebra \textbf{318} (2007), no.~2, 669--679,
  \eprint{arxiv}{math.RT/0608321}.

\bibitem{mozgovoy_motivica}
\bysame, \emph{{O}n the motivic {D}onaldson-{T}homas invariants of quivers with
  potentials}, 2011, \eprint{arxiv}{1103.2902}.

\bibitem{mozgovoy_noncommutative}
Sergey Mozgovoy and Markus Reineke, \emph{{O}n the noncommutative
  {D}onaldson-{T}homas invariants arising from brane tilings}, Adv. Math.
  \textbf{223} (2010), no.~5, 1521--1544, \eprint{arxiv}{0809.0117}.

\bibitem{nagao_deriveda}
Kentaro Nagao, \emph{{D}erived categories of small toric {Calabi-Yau} 3-folds
  and counting invariants}, September 2008, \eprint{arxiv}{0809.2994}.

\bibitem{nagao_counting}
Kentaro Nagao and Hiraku Nakajima, \emph{{C}ounting invariant of perverse
  coherent sheaves and its wall-crossing}, September 2008,
  \eprint{arxiv}{0809.2992}.

\bibitem{pandharipande_curve}
R.~Pandharipande and R.~P. Thomas, \emph{{C}urve counting via stable pairs in
  the derived category}, Invent. Math. \textbf{178} (2009), no.~2, 407--447,
  \eprint{arxiv}{0707.2348}.

\bibitem{reineke_harder-narasimhan}
Markus Reineke, \emph{{T}he {H}arder-{N}arasimhan system in quantum groups and
  cohomology of quiver moduli}, Invent. Math. \textbf{152} (2003), no.~2,
  349--368, \eprint{arxiv}{math/0204059}.

\bibitem{reineke_cohomology}
\bysame, \emph{{C}ohomology of noncommutative {H}ilbert schemes}, Algebr.
  Represent. Theory \textbf{8} (2005), no.~4, 541--561,
  \eprint{arxiv}{math.AG/0306185}.

\bibitem{reineke_counting}
\bysame, \emph{{C}ounting rational points of quiver moduli}, Int. Math. Res.
  Not. \textbf{17} (2006), ID 70456, \eprint{arxiv}{math/0505389}.

\bibitem{szendroi_non-commutative}
Balazs Szendroi, \emph{{N}on-commutative {D}onaldson-{T}homas theory and the
  conifold}, Geom.Topol. \textbf{12} (2008), 1171--1202,
  \eprint{arxiv}{0705.3419}.

\bibitem{thaddeus_stable}
Michael Thaddeus, \emph{{S}table pairs, linear systems and the {V}erlinde
  formula}, Invent. Math. \textbf{117} (1994), no.~2, 317--353,
  \eprint{arxiv}{alg-geom/9210007}.

\bibitem{toda_curve}
Yukinobu Toda, \emph{{C}urve counting theories via stable objects {I}.
  {DT}/{PT} correspondence}, J. Amer. Math. Soc. \textbf{23} (2010), no.~4,
  1119--1157, \eprint{arxiv}{0902.4371}.

\bibitem{young_computing}
Benjamin Young, \emph{{C}omputing a pyramid partition generating function with
  dimer shuffling}, J. Combin. Theory Ser. A \textbf{116} (2009), no.~2,
  334--350, \eprint{arxiv}{0709.3079}.

\bibitem{young_generating}
\bysame, \emph{{G}enerating functions for colored 3{D} {Y}oung diagrams and the
  {D}onaldson-{T}homas invariants of orbifolds}, Duke Math. J. \textbf{152}
  (2010), no.~1, 115--153, \eprint{arxiv}{0802.3948}, With an appendix by Jim
  Bryan.

\end{thebibliography}
\bibliographystyle{../tex/hamsplain}

\end{document}